\documentclass[12pt,leqno]{amsart}
\usepackage{euscript, amssymb, amsmath, amsthm}
\usepackage{epsfig}
\usepackage{graphicx}
\usepackage{caption}
\usepackage{longtable}
\usepackage{dcolumn}
\usepackage{setspace}
\usepackage[most]{tcolorbox}
\definecolor{webred}{rgb}{0.75,0,0}
\definecolor{webgreen}{rgb}{0,0.75,0}
\usepackage[citecolor=webgreen,colorlinks=true,linkcolor=webred]{hyperref}
\definecolor{refkey}{gray}{0.75}

\numberwithin{equation}{section}

\newtheorem{theo}{Theorem}[section]
\newtheorem{lem}{Lemma}[section]

\newtheorem{Def}[theo]{Definition}
\theoremstyle{remark}
\newtheorem{rem}{Remark}[section]

\newcommand{\de}{\delta}

\newcommand{\ep}{\varepsilon}
\def\R{{\mathbb{R}}}

\def\d{\displaystyle}
\def\e{{\varepsilon}}
\def\p{\partial}
\def\al{\alpha}

\date{}

\subjclass[2010]{35L71,  35B44}
\keywords{blow-up, lifespan, nonlinear wave equations, scale-invariant damping, time-derivative nonlinearity.}

\tcbset{
    frame code={}
    center title,
    left=0pt,
    right=0pt,
    top=0pt,
    bottom=0pt,
    colback=gray!10,
    colframe=white,
    width=\dimexpr\textwidth\relax,
    enlarge left by=0mm,
    boxsep=5pt,
    arc=0pt,outer arc=0pt,
    }

\begin{document}

\title[A blow-up result for the wave equation  with localized initial data]{A blow-up result for the wave equation  with localized initial data: the  scale-invariant damping and mass term with combined nonlinearities}
\author[M. Hamouda  and M. A. Hamza]{Makram Hamouda$^{1}$ and Mohamed Ali Hamza$^{1}$}
\address{$^{1}$ Basic Sciences Department, Deanship of Preparatory Year and Supporting Studies, P. O. Box 1982, Imam Abdulrahman Bin Faisal University, Dammam, KSA.}

\medskip

\email{mmhamouda@iau.edu.sa (M. Hamouda)} 
\email{mahamza@iau.edu.sa (M.A. Hamza)}

\pagestyle{plain}


\maketitle
\begin{abstract}
We are interested  in this article in studying the  damped wave equation with localized initial data,  in the \textit{scale-invariant case} with mass term and two combined nonlinearities. More precisely, we consider the following equation:
\begin{displaymath}
\d (E) \hspace{1cm} u_{tt}-\Delta u+\frac{\mu}{1+t}u_t+\frac{\nu^2}{(1+t)^2}u=|u_t|^p+|u|^q,
\quad \mbox{in}\ \R^N\times[0,\infty),
\end{displaymath}
with small initial data. Under some assumptions on the mass and damping coefficients, $\nu$ and $\mu>0$, respectively, we show that blow-up region and the lifespan bound of the solution of $(E)$ remain the same as  the ones obtained in \cite{Our2} in the case of a mass-free wave equation, {\it i.e.} $(E)$ with $\nu=0$. 
 Furthermore, using in part the computations done for $(E)$, we enhance the result in \cite{Palmieri}  on the Glassey conjecture for the solution of $(E)$ with omitting the nonlinear term $|u|^q$. Indeed,  the blow-up region is extended from $p \in (1, p_G(N+\sigma)]$, where $\sigma$ is given by \eqref{sigma} below, to $p \in (1, p_G(N+\mu)]$ yielding, hence, a better estimate of the lifespan when $(\mu-1)^2-4\nu^2<1$. Otherwise, the two results coincide. Finally, we may conclude that the mass term {\it has no influence} on the dynamics of $(E)$ (resp.  $(E)$ without the nonlinear term $|u|^q$), and the conjecture we made in \cite{Our2} on the threshold between the blow-up and the global existence regions obtained    holds true here.
\end{abstract}


\section{Introduction}
\par\quad

We consider the following  family of semilinear damped wave equations
\begin{equation}
\label{G-sys}
\left\{
\begin{array}{l}
\d u_{tt}-\Delta u+\frac{\mu}{1+t}u_t+\frac{\nu^2}{(1+t)^2}u=a|u_t|^p+b|u|^q,
\quad \mbox{in}\ \R^N\times[0,\infty),\\
u(x,0)=\e f(x),\ u_t(x,0)=\e g(x), \quad  x\in\R^N,
\end{array}
\right.
\end{equation}
where $a$ and $b$ are nonnegative constants and $\mu, \nu \ge 0$. The parameter $\e$ is a positive number which is characterizing the smallness of the initial
data,  and $f$ and $g$ are two compactly supported non-negative functions   on  $B_{\R^N}(0,R), R>0$.

We assume along this article  that $p, q>1$ and $q \le \frac{2N}{N-2}$ if $N \ge 3$.

The  linear equation  associated with \eqref{G-sys} reads as follows:
\begin{equation}\label{1.2}
\d u^L_{tt}-\Delta u^L+\frac{\mu}{1+t}u^L_t+\frac{\nu^2}{(1+t)^2}u^L=0.
\end{equation} 
It is clear that the equation \eqref{1.2} is invariant under the  transform
$$\tilde{u}^L(x,t)=u^L(\Omega x, \Omega(1+t)-1), \ \Omega>0.$$
This  explains  somehow the name of \textit{scale-invariant} case for \eqref{G-sys}. Obviously, one can apply two types of transformation to \eqref{1.2} leading to whether a purely damped wave equation or a wave equation with mass term. For the analysis of these cases, we introduce the parameter $\de$ defined as
\begin{equation}\label{delta}
\de=(\mu-1)^2-4\nu^2.
\end{equation}
It is worth mentioning to recall that the scale-invariant damping is the critical case between the class of parabolic equations (for $\de$ large enough) and the one of hyperbolic equations (when $\de$ is small). Note that the parameter $\de$ has an important role in the dynamics of the solution of \eqref{1.2} and consequently \eqref{G-sys}, see e.g. \cite{NWN,Palmieri-phd}.\\
Indeed, for $\de \ge 0$, by setting
\begin{equation}\label{transf-v}
u^L(x,t)=(1+t)^{-\alpha}v^L(x,t),
\end{equation}
where 
\begin{equation}\label{alpha}
\al = \frac{\mu -1 -\sqrt{\de}}{2} \ \text{which verifies} \ \al^2-(\mu -1) \al+\nu^2=0,
\end{equation}
then the obtained equation for $v^L$ is a damped wave equation (without mass term), which reads as follows:
\begin{equation}\label{1.2-v}
\d v^L_{tt}-\Delta v^L+\frac{1+\sqrt{\de}}{1+t}v^L_t=0.
\end{equation}
 However, for $\de < 0$, the situation is different and we introduce the Liouville transform:
$$\d w^L(x,t)=(1+t)^{\frac{\mu}{2}}u^L(x,t),$$
with $w^L$ satisfies a free-damped wave equation with mass term
\begin{equation}\label{1.2-w}
\d w^L_{tt}-\Delta w^L+\frac{1-\de}{4(1+t)^2}w^L=0.
\end{equation}

\subsection{The non perturbed case}

Let $\mu =\nu = 0$ throughout this subsection.  Then, by taking
$(a,b)=(0,1)$ in \eqref{G-sys},  the equation \eqref{G-sys} reduces to  the classical semilinear wave equation in relationship with the Strauss conjecture for which we recall the critical power  $q_S$  which is solution of 
\begin{equation}
(N-1)q^2-(N+1)q-2=0,
\end{equation}
and which is given  by
\begin{equation}
q_S=q_S(N):=\frac{N+1+\sqrt{N^2+10N-7}}{2(N-1)}.
\end{equation}
For $q \le q_S$ and  under suitable sign assumptions for the initial data,  there is no global solution for  \eqref{G-sys}, and for $q > q_S$ the existence of a global solution is ensured for small initial data; see e.g. \cite{John2,Strauss,YZ06,Zhou} among other references. \\

Now,  in the case  $(a,b)=(1,0)$, the Glassey conjecture yields the critical power $p_G$ which is given by
\begin{equation}\label{Glassey}
p_G=p_G(N):=1+\frac{2}{N-1}.
\end{equation}
The critical power, $p_G$, separates the two regions for the power $p$ characterized by the global existence (for $p>p_G$) and the nonexistence (for $p \le p_G$) of a global  solution under the smallness of the initial data; see e.g. \cite{Hidano1,Hidano2,John1,Rammaha,Sideris,Tzvetkov,Zhou1}.\\

Here, we are interested in  the case   $a, b \neq 0$, thus, we may assume that $(a,b)=(1,1)$. For this case and when   the powers $p$ and $q$ satisfy   $p \le p_G$ or $q \le q_S$, the blow-up of the solution of  \eqref{G-sys} can be similarly obtained. However, for $p > p_G$ and $q > q_S$, there is a new blow-up border which is characterized by
\begin{equation}\label{1.5}
\lambda(p, q, N):=(q-1)\left((N-1)p-2\right) < 4.
\end{equation}
The reader may consult \cite{Dai,Han-Zhou,Hidano3,Wang} for more details.\\
It is proven in \cite{Hidano3} that,  for $p > p_G$ and $q > q_S$,  \eqref{1.5} implies the global existence of the solution of \eqref{G-sys} (with $\mu =\nu = 0$ and $(a,b)=(1,1)$). This is specific to the case of mixed  nonlinearities. Therefore, it is interesting to see if this phenomenon still occurs for the damping case $\mu > 0$. This will be exposed in the next subsection.

\subsection{The scale-invariant damped   case}
We consider here $\mu > 0$, $\nu=0$ and $(a,b)=(0,1)$. Hence, for $\mu$  large enough,  the equation \eqref{G-sys} is of a parabolic type, namely it behaves like a heat-type equation; see e.g. \cite{dabbicco1,dabbicco2,wakasugi}. However, for small  $\mu$,  the solution of \eqref{G-sys} is  like a wave. In fact, the damping has a shifting effect by $\mu>0$ on the critical power $q_S$, and more precisely we have the blow-up for
$$0<\mu < \frac{N^2+N+2}{N+2} \quad \text{and} \quad 1<q\le q_S(N+\mu);$$
see e.g. \cite{Palmieri-Reissig, Palmieri, Palmieri-Tu, Tu-Lin1, Tu-Lin}, and \cite{Dab1,Dab2} for the case 
 $\mu=2$ and $N=2,3$. The global existence for $\mu=2$ is proven in \cite{Dab1, Dab2, Palmieri2}.

Now, for  $(a,b)=(1,0)$, we first mention the blow-up result for the solution of \eqref{G-sys} (with $(a,b)=(1,0)$) obtained by Lai and Takamura in \cite{LT2} where  
a first estimate of the lifespan upper bound  is given. Then,  Palmieri and Tu  improved this result in \cite{Palmieri} by extending the blow-up region for $p$ in the system  \eqref{G-sys} with $(a,b)=(1,0)$, one time-derivative nonlinearity ({\it i.e.} \eqref{T-sys-bis} below) and a mass term. More precisely, they obtain a blow-up result for $p \in (1, p_G(N+\sigma(\mu,0))]$ where 
\begin{equation}\label{sigma}
\sigma=\sigma(\mu,\nu):=\left\{
\begin{array}{lll}
\mu+1 -\sqrt{\de}& \textnormal{if} &\de \in [0,1),\\
\mu & \text{if} &\de \ge 1,
\end{array}
\right.
\end{equation}
and $\de$ is given by \eqref{delta}.  Nevertheless, the result in \cite{Palmieri} was recently refined in \cite{Our2} by extending the upper bound for $p$ from $p_G(N+\sigma(\mu,0)$ to $p_G(N+\mu)$.  Obviously, the result in \cite{Our2} improves the one in \cite{Palmieri} only for $\mu \in (0,2)$.

Finally, in the presence of two mixed nonlinearities, {\it i.e.} $(a,b)=(1,1)$, it is proved in \cite{Our,Our2} that the blow-up region  of the solution of  \eqref{G-sys}, in this case, is in fact a shift by $\mu$ of the one related to the same problem but without damping. Obviously,  for $p \le p_G(N+\mu)$ and $q \le q_S(N+\mu)$, the  blow-up of the solution of \eqref{G-sys} can be easily obtained.   Furthermore, for $p > p_G(N+\mu)$, $q > q_S(N+\mu)$  and the combination of a weak damping term and  two mixed nonlinearities, the blow-up bound becomes $\lambda(p, q, N+\mu)<4$ instead of \eqref{1.5} which characterizes the free-damping case.

\subsection{The scale-invariant damping and mass   case}

Along this part, we assume that $\mu > 0$ and $\nu>0$. Therefore, let us start with the case $(a,b)=(0,1)$. It is known in the literature  that  the mass and the scale-invariant damping terms are in competition generating thus several cases depending on the values of $\mu$ and $\nu$, see e.g. \cite{Palmieri-phd}. More precisely, as mentioned for the linear equation \eqref{1.2}, one can recall that the mass term steps in the dynamics for $\de \ge 0$. Indeed, for $\de \ge (N+1)^2$, which corresponds to the large values of $\de$ and consequently the large values of  the damping term $\mu$, it is proven that the critical exponent is the shifted Fujita exponent $q_F(N+\frac{\mu-1-\sqrt{\de}}{2})$ where $q_F(N)=1+\frac{2}{N}$, see \cite{NPR2017,Palmieri-2018,Palmieri-2019-2,PR-2018}. However, for $\de \in [0,1)$ (corresponding to the small values of $\de$), the authors in \cite{Palmieri-Reissig} show  the appearance of a competition between the Fujita and the Strauss exponents. Indeed, they obtained a blow-up result for $q \le \max(q_F(N+\frac{\mu-1-\sqrt{\de}}{2}),q_S(N+\mu))$. We note that a recent improvement and a better comprehension of the transition from the heat-like equation to the wave-like one are obtained in \cite{LST-2020}. On the other hand, a blow-up result is proven in \cite{Palmieri-Tu} for all $\de \ge 0$ and $q \le q_S(N+\mu)$. Nevertheless, for  $\de < 0$, the situation is much different leading to the well-known Klein-Gordon equation where the mass term is more influent, and to the best of our knowledge the dynamics are less understood in the literature; see e.g. \cite{NWN}.

In this article, we consider the following Cauchy problem in 
the scale-invariant case and combined nonlinearities:
\begin{equation}
\label{T-sys}
\left\{
\begin{array}{l}
\d u_{tt}-\Delta u+\frac{\mu}{1+t}u_t+\frac{\nu^2}{(1+t)^2}u=|u_t|^p+|u|^q, 
\quad \mbox{in}\ \R^N\times[0,\infty),\\
u(x,0)=\e f(x),\ u_t(x,0)=\e g(x), \quad  x\in\R^N,
\end{array}
\right.
\end{equation}
where $\mu, \nu^2 > 0$, $\ N \ge 1$, $\e>0$ is a sufficiently  small parameter,
and  $f,g$ are compactly supported non-negative functions on which some assumptions will be specified later on.

One of the objectives in the present work is to study of the Cauchy problem (\ref{T-sys}) for $\mu, \nu^2>0$ and the influence of the parameters $\mu$ and $\nu$ on the blow-up result and the lifespan estimate.  Indeed, thanks to the transform \eqref{transf-v}, for $\de \ge 0$, and a better comprehension of the linear problem corresponding to (\ref{T-sys}), we surprisingly show that {\it there is no influence of the mass term} in the blow-up dynamics of the solution of  (\ref{T-sys}).

Moreover, under the same hypotheses on the data as for (\ref{T-sys}), we are interested now in studying the following equation which is characterized by the presence of a one  nonlinearity of time-derivative type, namely 
\begin{equation}
\label{T-sys-bis}
\left\{
\begin{array}{l}
\d u_{tt}-\Delta u+\frac{\mu}{1+t}u_t+\frac{\nu^2}{(1+t)^2}u=|u_t|^p, 
\quad \mbox{in}\ \R^N\times[0,\infty),\\
u(x,0)=\e f(x),\ u_t(x,0)=\e g(x), \quad  x\in\R^N.
\end{array}
\right.
\end{equation}
Using the computations obtained for (\ref{T-sys}), we will enhance the blow-up interval, $p \in (1, p_G(N+\sigma)]$ ($\sigma$ is given by \eqref{sigma}),  proven in \cite{Palmieri},  to  arrive at the interval $p \in (1, p_G(N+\mu)]$, for $\de \in (0,1)$. Nevertheless, for $\de \ge 1$, our result for (\ref{T-sys-bis}) coincides  with the one in \cite{Palmieri}. Inspired from \cite{Our2}, we may conjecture here again that the obtained upper bound exponent is the critical one in the sense that it separates the blow-up and the global existence regions. Notice that our method is different form the one  in \cite{Palmieri} where the use of an integral representation of the solution is employed. However, in the present work, we make use of the multiplier technique together with the fact  that $G_2(t)$ is  coercive  starting from relatively large time thanks to the presence of 
the nonlinearity $|u_t|^p$ which controls in part the negativity of $G_2(t)$.
\\

The  rest of the article is organized as follows. First, Section \ref{sec-main} is devoted to the  definition of  the weak formulation of (\ref{T-sys}), in the energy space, together with the statement of the main theorems of our work.  Then, we prove in Section \ref{aux} some technical lemmas. These auxiliary results, among other tools, are used to conclude  the proof of the main results in Sections \ref{proof} and \ref{sec-ut}.  Indeed, in Section \ref{proof} (resp. Sec. \ref{sec-ut}), we prove the blow-up of the solution of (\ref{T-sys}) (resp. (\ref{T-sys-bis})) for $p$ and $q$ satisfying $\lambda(p, q, N+\mu)<4$ (resp. for $p$ verifying $p \in (1, p_G(N+\mu)]$).

\section{Main Results}\label{sec-main}
\par

In this section, we will   state the main results in this work. To this end, we first give a sense to  the  solution of (\ref{T-sys}) in the corresponding energy space. Hence, the weak formulation of  (\ref{T-sys}) reads as:
\begin{Def}\label{def1}
 We call $u$ is a weak  solution of
 (\ref{T-sys}) on $[0,T)$
if
\begin{displaymath}
\left\{
\begin{array}{l}
u\in \mathcal{C}([0,T),H^1(\R^N))\cap \mathcal{C}^1([0,T),L^2(\R^N)), \vspace{.1cm}\\
 u \in L^q_{loc}((0,T)\times \R^N) \ \text{and} \ u_t \in L^p_{loc}((0,T)\times \R^N),
 \end{array}
  \right.
\end{displaymath}
verifies, for all $\Phi\in \mathcal{C}_0^{\infty}(\R^N\times[0,T))$ and all $t\in[0,T)$, the following identity:
\begin{equation}
\label{energysol2}
\begin{array}{l}
\d\int_{\R^N}u_t(x,t)\Phi(x,t)dx-\int_{\R^N}u_t(x,0)\Phi(x,0)dx  -\int_0^t  \int_{\R^N}u_t(x,s)\Phi_t(x,s)dx \,ds\vspace{.2cm}\\
\d+\int_0^t  \int_{\R^N}\nabla u(x,s)\cdot\nabla\Phi(x,s) dx \,ds+\int_0^t  \int_{\R^N}\frac{\mu}{1+s}u_t(x,s) \Phi(x,s)dx \,ds\vspace{.2cm}\\
\d  +\int_0^t  \int_{\R^N}\frac{\nu^2}{(1+s)^2}u(x,s) \Phi(x,s)dx \,ds=\int_0^t \int_{\R^N}\left\{|u_t(x,s)|^p+|u(x,s)|^q\right\}\Phi(x,s)dx \,ds.
\end{array}
\end{equation}
\end{Def}

Of course the weak formulation corresponding to (\ref{T-sys-bis}) can be also obtained by \eqref{energysol2} without the nonlinear term $|u|^q$ and with the necessary modifications.

Hence, with the help of  the multiplier $m(t)$ defined by
\begin{equation}
\label{test1}
m(t):=(1+t)^{\mu},
\end{equation}
we can rewrite Definition \ref{def1}, by considering $m(t)\Phi(x,t)$ as a test function, in  the following equivalent formulation.
\begin{Def}\label{def2}
 We say that $u$ is a weak  solution of
 (\ref{T-sys}) on $[0,T)$
if
\begin{displaymath}
\left\{
\begin{array}{l}
u\in \mathcal{C}([0,T),H^1(\R^N))\cap \mathcal{C}^1([0,T),L^2(\R^N)), \vspace{.1cm}\\
 u \in L^q_{loc}((0,T)\times \R^N) \ \text{and} \ u_t \in L^p_{loc}((0,T)\times \R^N),
 \end{array}
  \right.
\end{displaymath}
satisfies, for all $\Phi\in \mathcal{C}_0^{\infty}(\R^N\times[0,T))$ and all $t\in[0,T)$, the following equation:
\begin{equation}
\label{energysol}
\begin{array}{l}
m(t)\d\int_{\R^N}u_t(x,t)\Phi(x,t)dx-\int_{\R^N}u_t(x,0)\Phi(x,0)dx \vspace{.2cm}\\
\d -\int_0^t m(s) \int_{\R^N}u_t(x,s)\Phi_t(x,s)dx \,ds+\int_0^t m(s) \int_{\R^N}\nabla u(x,s)\cdot\nabla\Phi(x,s) dx \,ds\vspace{.2cm}\\
\d+\int_0^t  \int_{\R^N}\frac{\nu^2m(s)}{(1+s)^2}u(x,s) \Phi(x,s)dx \,ds\vspace{.2cm}\\
\d=\int_0^t m(s) \int_{\R^N}\left\{|u_t(x,s)|^p+|u(x,s)|^q\right\}\Phi(x,s)dx \,ds.
\end{array}
\end{equation}
\end{Def}

In the following, we will state the main results in this article.
\begin{theo}
\label{blowup}
Let $p, q>1$, $\nu^2, \mu \ge 0$  and $\de \ge 0$ such that 
\begin{equation}\label{assump}
\lambda(p, q, N+\mu)<4,
\end{equation}
where  $\lambda$ is given by \eqref{1.5}, and $p>p_G(N+\mu)$, $q>q_S(N+\mu)$. Furthermore, assume that  $f\in H^1(\R^N)$ and $g\in L^2(\R^N)$ are non-negative functions which are compactly supported on  $B_{\R^N}(0,R)$,
  do not vanish everywhere and satisfy
  \begin{equation}\label{hypfg}
  \frac{\mu-1-\sqrt{\de}}{2}f(x)+g(x) > 0.
  \end{equation}
Let $u$ be an energy solution of \eqref{T-sys} on $[0,T_\e)$ such that $\mbox{\rm supp}(u)\ \subset\{(x,t)\in\R^N\times[0,\infty): |x|\le t+R\}$. 
Then, there exists a constant $\e_0=\e_0(f,g,N,R,p,q,\mu,\nu)>0$
such that $T_\e$ verifies
\[
T_\e\leq
 C \,\e^{-\frac{2p(q-1)}{4-\lambda(p, q, N+\mu)}},
\]
 where $C$ is a positive constant independent of $\e$ and $0<\e\le\e_0$.
\end{theo}

\begin{theo}
\label{th_u_t}
Let $\nu^2, \mu \ge 0$  and $\de \ge 0$. Assume that  $f\in H^1(\R^N)$ and $g\in L^2(\R^N)$ are non-negative and compactly supported functions   on  $B_{\R^N}(0,R)$ which
  do not vanish everywhere and verify \eqref{hypfg}. Let $u$ be an energy solution of \eqref{T-sys-bis} on $[0,T_\e)$ such that $\mbox{\rm supp}(u)\ \subset\{(x,t)\in\R^N\times[0,\infty): |x|\le t+R\}$. 
Then, there exists a constant $\e_0=\e_0(f,g,N,R,p,\mu,\nu)>0$
such that $T_\e$ verifies
\begin{displaymath}
T_\e \leq
\d \left\{
\begin{array}{ll}
 C \, \e^{-\frac{2(p-1)}{2-(N+\mu-1)(p-1)}}
 &
 \ \text{for} \
 1<p<p_G(N+\mu), \vspace{.1cm}
 \\
 \exp\left(C\e^{-(p-1)}\right)
&
 \ \text{for} \ p=p_G(N+\mu),
\end{array}
\right.
\end{displaymath}
 where $C$ is a positive constant independent of $\e$ and $0<\e\le\e_0$.
\end{theo}

\begin{rem}
We note that the result in Theorem \ref{blowup} does not depend on the parameter $\nu$. Hence, thanks to \cite[Remark 2.3]{Our2}, we have  the existence of a pair $(p_0(N+\mu),q_0(N+\mu))$ which satisfies  \eqref{assump}, $p_0(N+\mu) > p_G(N+\mu)$ and $q_0(N+\mu) > q_S(N+\mu)$. Consequently, the hypothesis on $p$ and $q$ in Theorem \ref{blowup} makes sense.
\end{rem}

\begin{rem}
It is clear  that the limiting value $p_G(N+\mu)$ is less or equal to the critical exponent for $p$ in Theorem \ref{th_u_t} and the blow-up result there does not depend on the parameter $\nu$. Hence, we believe, as observed in \cite[Remark 2.1]{Our2}, that this limiting value is the critical one. The rigorous proof of this assertion (which is related to the global existence) will be the subject of a forthcoming work.
\end{rem}

\begin{rem}
We note that for $q \le q_S(N+\mu)$ and $p \le p_G(N+\mu)$ ($\de \ge 0$) a blow-up result for \eqref{T-sys} is proven in \cite{Palmieri-Tu} and \cite{Our2}, respectively. Moreover, as explained before, the presence of two mixed nonlinearities in \eqref{T-sys} generates  a new  region in both cases  $\mu=0$ and  $\mu>0$; see \cite{Hidano3} and \cite{Our2}, respectively. Hence, we concentrate our effort in the present work to look for the blow-up in the region $q > q_S(N+\mu)$ and $p > p_G(N+\mu)$; this justifies the hypotheses on $p$ and $q$ in Theorem \ref{blowup}.
\end{rem}

\section{Some auxiliary results}\label{aux}
\par

First, we introduce the  positive test function  $\psi(x,t)$ which is defined by
\begin{equation}
\label{test11}
\psi(x,t):=\rho(t)\phi(x);
\quad
\phi(x):=
\left\{
\begin{array}{ll}
\d\int_{S^{N-1}}e^{x\cdot\omega}d\omega & \mbox{for}\ N\ge2,\vspace{.2cm}\\
e^x+e^{-x} & \mbox{for}\  N=1,
\end{array}
\right.
\end{equation}
where $\phi(x)$ is introduced in \cite{YZ06}   and $\rho(t)$, \cite{Palmieri1,Palmieri-Tu,Tu-Lin1,Tu-Lin},   is solution of 
\begin{equation}\label{lambda}
\frac{d^2 \rho(t)}{dt^2}-\rho(t)-\frac{d}{dt}\left(\frac{\mu}{1+t}\rho(t)\right)+\frac{\nu^2}{(1+t)^2}\rho(t)=0.
\end{equation}
Then, the expression of  $\rho(t)$ reads as follows:
\begin{equation}\label{lmabdaK}
\rho(t)=(t+1)^{\frac{\mu+1}{2}}K_{\frac{\sqrt{\de}}{2}}(t+1),
\end{equation}
where 
$$K_{\xi}(t)=\int_0^\infty\exp(-t\cosh \zeta)\cosh(\xi \zeta)d\zeta,\ \xi\in \mathbb{R}.$$

Using for example the equation (18) in \cite{Palmieri1}   (see also  the proof of Lemma 2.1 in \cite{Tu-Lin} (with $\eta=1)$ or \cite{Palmieri-Tu}), we have 
\begin{equation}\label{eqK}
K'_{\xi}(t)=-K_{\xi+1}(t)+\frac{\xi}{t}K_{\xi}(t).
\end{equation}
Hence, we infer that
\begin{equation}\label{lambda'lambda}
\frac{\rho'(t)}{\rho(t)}=\frac{\mu+1+\sqrt{\de}}{2(1+t)}-\frac{K_{\frac{\sqrt{\de}}2 +1}(t+1)}{K_{\frac{\sqrt{\de}}2}(t+1)}.
\end{equation}
From \cite{Gaunt}, we have the following property for the function $K_{\xi}(t)$:
\begin{equation}\label{Kmu}
K_{\xi}(t)=\sqrt{\frac{\pi}{2t}}e^{-t} (1+O(t^{-1}), \quad \text{as} \ t \to \infty.
\end{equation}
Combining \eqref{lambda'lambda} and \eqref{Kmu}, we infer that
\begin{equation}\label{lambda'lambda1}
\frac{\rho'(t)}{\rho(t)}=-1+O(t^{-1}), \quad \text{as} \ t \to \infty.
\end{equation}

Finally, we refer the reader to \cite{Erdelyi} for more details about the properties of the function $K_{\mu}(t)$.
Moreover, the function $\phi(x)$ verifies $\Delta\phi=\phi$.\\
Note that the function $\psi(x, t)$ satisfies the conjugate equation corresponding to \eqref{1.2}, namely we have
\begin{equation}\label{lambda-eq}
\partial^2_t \psi(x, t)-\Delta \psi(x, t) -\frac{\partial}{\partial t}\left(\frac{\mu}{1+t}\psi(x, t)\right)+\frac{\nu^2}{(1+t)^2}\psi(x, t)=0.
\end{equation}

Along  this article, we will denote by $C$  a generic positive constant which may depend on the data ($p,q,\mu,\nu,N,f,g,\ep_0$) but not on $\ep$, and whose value may change from line to line. However, the dependence of the constant $C$ may be described when needed depending on the context.\\

The following lemma  holds true for the function $\psi(x, t)$.
\begin{lem}[\cite{YZ06}]
\label{lem1} Let  $r>1$.
There exists a constant $C=C(N,R,p,r)>0$ such that
\begin{equation}
\label{psi_0}
\int_{|x|\leq t+R}\Big(\psi_0(x,t)\Big)^{r}dx
\leq C(1+t)^{\frac{(2-r)(N-1)}{2}},
\quad\forall \ t\ge0,
\end{equation}
where $\psi_0(x, t):=e^{-t}\phi(x)$, and $\phi(x)$ is given by \eqref{test11}, and
\begin{equation}
\label{psi}
\int_{|x|\leq t+R}\Big(\psi(x,t)\Big)^{r}dx
\leq C\rho^r(t)e^{rt}(1+t)^{\frac{(2-r)(N-1)}{2}},
\quad\forall \ t\ge0,
\end{equation}
where $\psi(x, t)$  is given by \eqref{test11}.
\end{lem}

\par
Now, we define here the functionals that we will use to prove the blow-up criteria later on:
\begin{equation}
\label{F1def}
G_1(t):=\int_{\R^N}u(x, t)\psi(x, t)dx,
\end{equation}
and
\begin{equation}
\label{F2def}
G_2(t):=\int_{\R^N}\p_tu(x, t)\psi(x, t)dx.
\end{equation}

We aim in the following to show that the functionals $G_1(t)$ and $G_2(t)$ are coercive. This will be the first observation that we will use later on to improve the main results of this article. We note here that the proof of Lemma \ref{F1} below is known in the literature; see e.g. \cite{Palmieri1,Tu-Lin1,Tu-Lin}. However, for a later use of some computations in the proof of Lemma \ref{F1}, we choose to detail the steps therein. Nevertheless, Lemmas \ref{F1-2} and \ref{F11} constitute somehow a novelty in this work and their utilization in the proofs of Theorems \ref{blowup}  and  \ref{th_u_t} is fundamental. 

Let us stress out in what follows the particularity of the dynamics of $G_2(t)$ in terms of time $t$. The first observation consists in showing that $G_2(t)$ possesses a negative lower bound; see Lemma \ref{F1-2}  below. Then, the second property states that $G_2(t)$ is coercive starting from relatively large time which is growing as $\ep$ is approaching zero; see Lemma \ref{F11}  below.  Hence, in comparison with our previous work \cite{Our2}, where we studied the same problem \eqref{T-sys} without the mass term ($\nu=0$), we note here that the functional $G_2(t)$, while dealing with the mass term, exhibits a different behavior. More precisely, we remark that $G_2(t)$ starts with a positive value (since the initial data are positive) and then it may take some negative values, maybe several times, to end up with the coercive characteristics for large time. However, the functional $G_1(t)$  is coercive starting from a positive finite time which is independent of $\ep$.

\begin{lem}
\label{F1}
Let $u$ be an energy solution of the system \eqref{T-sys} with initial data satisfying the assumptions in Theorem \ref{blowup}. Then, there exists $T_0=T_0(\mu,\nu)>1$ such that 
\begin{equation}
\label{F1postive}
G_1(t)\ge C_{G_1}\, \e, 
\quad\text{for all}\ t \ge T_0,
\end{equation}
where $C_{G_1}$ is a positive constant which may depend on $f$, $g$, $N, \mu$ and $\nu$.
\end{lem}
\begin{proof} 
Let $t \in [0,T)$.  Using Definition \ref{def1} and  performing an integration by parts in space in the fourth term in the left-hand side of \eqref{energysol2}, we obtain
\begin{equation}\label{eq3}
\begin{array}{l}
\d \int_{\R^N}u_t(x,t)\Phi(x,t)dx
-\e\int_{\R^N}g(x)\Phi(x,0)dx -\int_0^t\int_{\R^N}\left\{
u_t(x,s)\Phi_t(x,s)+u(x,s)\Delta\Phi(x,s)\right\}dx \, ds\vspace{.2cm}\\
\d +\int_0^t  \int_{\R^N}\frac{\mu}{1+s}u_t(x,s) \Phi(x,s)dx \,ds+\int_0^t  \int_{\R^N}\frac{\nu^2}{(1+s)^2}u(x,s) \Phi(x,s)dx \,ds\vspace{.2cm}\\
\d=\int_0^t\int_{\R^N}\left\{|u_t(x,s)|^p+|u(x,s)|^q\right\}\Phi(x,s)dx \, ds, \quad \forall \ \Phi\in \mathcal{C}_0^{\infty}(\R^N\times[0,T)).
\end{array}
\end{equation}
Substituting in \eqref{eq3} $\Phi(x, t)$ by $\psi(x, t)$, performing an integration by parts for third term in the first line and the first term in the second line of \eqref{eq3} and utilizing \eqref{test11} and \eqref{lambda-eq}, we obtain
\begin{equation}
\begin{array}{l}\label{eq5}
\d \int_{\R^N}\big[u_t(x,t)\psi(x,t)- u(x,t)\psi_t(x,t)+\frac{\mu}{1+t}u(x,t) \psi(x,t)\big]dx
\vspace{.2cm}\\
\d=\int_0^t\int_{\R^N}\left\{|u_t(x,s)|^p+|u(x,s)|^q\right\}\psi(x,s)dx \, ds 
+\d \e \, C(f,g),
\end{array}
\end{equation}
where 
\begin{equation}\label{Cfg}
C(f,g):=\int_{\R^N}\big[\big(\mu\rho(0)-\rho'(0)\big)f(x)+\rho(0)g(x)\big]\phi(x)dx.
\end{equation}
Using \eqref{lmabdaK}--\eqref{lambda'lambda}, we infer that
\begin{equation}\label{Cfg1}
\mu\rho(0)-\rho'(0)=\frac{\mu-1-\sqrt{\de}}{2}K_{\frac{\sqrt{\de}}{2}}(1)+K_{\frac{\sqrt{\de}}2 +1}(1).
\end{equation}
Hence,  we have
\begin{equation}\label{Cfg}
C(f,g)=K_{\frac{\sqrt{\de}}{2}}(1) \int_{\R^N} \big[\frac{\mu-1-\sqrt{\de}}{2}f(x)+g(x)\big]\phi(x)dx +K_{\frac{\sqrt{\de}}2 +1}(1)\int_{\R^N}g(x)\phi(x)dx.
\end{equation}
Thanks to \eqref{hypfg} we deduce  that the constant $C(f,g)$ is positive.\\
Hence, using the definition of $G_1$, as in \eqref{F1def},  and \eqref{test11}, the equation  \eqref{eq5} yields
\begin{equation}
\begin{array}{l}\label{eq6}
\d G_1'(t)+\Gamma(t)G_1(t)=\int_0^t\int_{\R^N}\left\{|u_t(x,s)|^p+|u(x,s)|^q\right\}\psi(x,s)dx \, ds +\e \, C(f,g),
\end{array}
\end{equation}
where 
\begin{equation}\label{gamma}
\Gamma(t):=\frac{\mu}{1+t}-2\frac{\rho'(t)}{\rho(t)}.
\end{equation}
Multiplying  \eqref{eq6} by $\frac{(1+t)^\mu}{\rho^2(t)}$ and integrating over $(0,t)$, we obtain
\begin{align}\label{est-G1}
 G_1(t)
\ge G_1(0)\frac{\rho^2(t)}{(1+t)^\mu}+{\e}C(f,g)\frac{\rho^2(t)}{(1+t)^\mu}\int_0^t\frac{(1+s)^\mu}{\rho^2(s)}ds.
\end{align}
Observing that $G_1(0)=\ep K_{\frac{\sqrt{\de}}{2}}(1)\int_{\R^N}f(x) \phi(x)dx>0$ and using \eqref{lmabdaK}, the estimate \eqref{est-G1} implies that
\begin{align}\label{est-G1-1}
 G_1(t)
\ge {\e}C(f,g)(1+t)K^2_{\frac{\sqrt{\de}}{2}}(t+1)\int^t_{t/2}\frac{1}{(1+s)K^2_{\frac{\sqrt{\de}}{2}}(s+1)}ds.
\end{align}
Using \eqref{Kmu}, we infer that there exists $T_0=T_0(\mu,\nu)>1$ such that 
\begin{align}\label{est-double}
(1+t)K^2_{\frac{\sqrt{\de}}{2}}(t+1)>\frac{\pi}{4} e^{-2(t+1)} \quad \text{and}  \quad (1+t)^{-1}K^{-2}_{\frac{\sqrt{\de}}{2}}(t+1)>\frac{1}{\pi} e^{2(t+1)}, \ \forall \ t \ge T_0/2.
\end{align}
Hence, we have
\begin{align}\label{est-G1-2}
 G_1(t)
\ge \frac{\e}{4}C(f,g)e^{-2t}\int^t_{t/2}e^{2s}ds\ge \frac{\e}{8}C(f,g)e^{-2t}(e^{2t}-e^{t}), \ \forall \ t \ge T_0.
\end{align}
Finally, using $e^{2t}>2e^{t}, \forall \ t \ge 1$, we deduce that
\begin{align}\label{est-G1-3}
 G_1(t)
\ge \frac{\e}{16}C(f,g), \ \forall \ t \ge T_0.
\end{align}

This ends the proof of Lemma \ref{F1}.
\end{proof}

As mentioned above, we think that the functional $G_2(t)$ cannot be nonnegative for all $t \ge 0$ (see the Appendix for some numerical simulations). 
However, we will prove in the following lemma that this functional possesses a negative lower bound independent of $\ep$.

\begin{lem}
\label{F1-2}
Assume the existence of an energy solution $u$ of the system \eqref{T-sys}  with initial data satisfying the hypotheses in Theorem \ref{blowup}. Then, for all $t \in (0,T)$, we have 
\begin{align}\label{G2+bis7}
\d G_2(t)+\mathcal{K}\nu^2\left\{ 1 +\nu^{\frac{2}{p-1}}e^{\frac{p}{p-1}t} (1+t)^{\frac{N-1}{2}} \right\} \ge  0,
\end{align}
where $\mathcal{K}$ is a positive constant which may depend on $p, f$, $g$, $N,R,  \ep_0$ and $\mu$ but not on $\ep$ and $\nu$\footnote{ We choose here to make explicit the dependence of the constant $\mathcal{K}$ on $\nu$ to point out the difference between the cases with and without the mass term.}.
\end{lem}
\begin{proof} 
Let $t \in [0,T)$. Then, using Definition \ref{def2},  performing an integration by parts in space in the fourth term in the left-hand side of \eqref{energysol} and choosing $\psi_0(x, t)$ as a test function\footnote{ Note that  it is possible to consider here not compactly supported test
functions thanks to the support property of $u$. Indeed, it is sufficient to replace $\psi_0(x, t)$ by $\psi_0(x, t) \chi(x, t)$ where $\chi$ is compactly supported such that $\chi(x, t)\equiv 1$ on $\mbox{\rm supp}(u)$.},  we infer that
\begin{equation}
\begin{array}{l}\label{eq4}
\d m(t)\int_{\R^N}u_t(x,t)\psi_0(x,t)dx
-\e\int_{\R^N}g(x)\psi_0(x,0)dx \vspace{.2cm}\\
\d+\int_0^tm(s)\int_{\R^N}\left\{
u_t(x,s)\psi_0(x,s)-u(x,s)\psi_0(x,s)\right\}dx \, ds \vspace{.2cm}\\
\d+\int_0^t  \int_{\R^N}\frac{\nu^2m(s)}{(1+s)^2}u(x,s) \psi_0(x,s)dx \,ds\vspace{.2cm}\\
\d=\int_0^tm(s)\int_{\R^N}\left\{|u_t(x,s)|^p+|u(x,s)|^q\right\}\psi_0(x,s)dx \, ds.
\end{array}
\end{equation}
We  introduce the following functionals
\begin{equation}
\label{F1def1}
F_1(t):=\int_{\R^N}u(x, t)\psi_0(x, t)dx, 
\end{equation}
and
\begin{equation}
\label{F2def1}
F_2(t):=\int_{\R^N}u_t(x, t)\psi_0(x, t)dx,
\end{equation}
where $\psi_0(x, t):=e^{-t}\phi(x)$, and $\phi(x)$ is given by \eqref{test11}.\\
Hence, using the definition of $F_1$ and the fact that $$\d \int_0^tm(s)F_1'(s) ds=- \int_0^tm'(s)F_1(s) ds+m(t)F_1(t)-F_1(0),$$ the equation  \eqref{eq4} yields
\begin{equation}
\begin{array}{l}\label{eq5-1}
\d m(t)(F_1'(t)+2F_1(t))
-{\e}C_0(f,g) +\int_0^t  \frac{\nu^2m(s)}{(1+s)^2}F_1(s) \,ds \vspace{.2cm}\\
\d=\int_0^tm'(s)F_1(s) ds+\int_0^tm(s)\int_{\R^N}\left\{|u_t(x,s)|^p+|u(x,s)|^q\right\}\psi_0(x,s)dx \, ds,
\end{array}
\end{equation}
where 
$$C_0(f,g):=\int_{\R^N}\left\{f(x)+g(x)\right\}\phi(x)dx.$$
Hence, using the definition of $F_1$ and  $F_2$, given respectively by \eqref{F1def} and  \eqref{F2def}, and the fact that
 \begin{equation}\label{def231}\d F_1'(t) +F_1(t)= F_2(t),\end{equation}
 the equation  \eqref{eq5-1} yields
\begin{equation}
\begin{array}{l}\label{eq5bis1}
\d m(t)(F_2(t)+F_1(t))
-{\e}C_0(f,g) +\int_0^t  \frac{\nu^2m(s)}{(1+s)^2}F_1(s) \,ds \vspace{.2cm}\\
\d=\int_0^tm'(s)F_1(s) ds+\int_0^tm(s)\int_{\R^N}\left\{|u_t(x,s)|^p+|u(x,s)|^q\right\}\psi_0(x,s)dx \, ds.
\end{array}
\end{equation}
Differentiating the  equation \eqref{eq5bis1} in time and using \eqref{def231}, we obtain
\begin{align}\label{F1+bis1}
\frac{d}{dt} \left\{F_2(t)m(t)\right\}+   2m(t)F_2(t)
= m(t)(F_1(t)+F_2(t))-\frac{\nu^2m(t)}{(1+t)^2}F_1(t)\\+m(t)\int_{\R^N}\left\{|u_t(x,t)|^p+|u(x,t)|^q\right\}\psi_0(x,t)dx.\nonumber
\end{align}
Using   \eqref{eq5bis1}, the identity \eqref{F1+bis1} becomes
\begin{align}\label{F1+bis3}
\begin{array}{l}
\d \frac{d}{dt} \left\{F_2(t)m(t)\right\}+   2m(t)F_2(t)
=\d  {\e}C_0(f,g) \vspace{.2cm}\\
\d +\int_0^tm(s)\int_{\R^N}\left\{|u_t(x,s)|^p+|u(x,s)|^q\right\}\psi_0(x,s)dx \, ds\vspace{.2cm}\\
\d +m(t)\int_{\R^N}\left\{|u_t(x,t)|^p+|u(x,t)|^q\right\}\psi_0(x,t)dx+\Sigma_1(t)+\nu^2\Sigma_2(t)+\nu^2\Sigma_3(t),
\end{array}
\end{align}
where $(m(t)=(1+t))^{\mu})$
\begin{equation}\label{sigma11}
\d \Sigma_1(t)=\d  \int_0^t m'(s)F_1(s) ds=\d \mu \int_0^t (1+s)^{\mu-1}F_1(s) ds,
\end{equation}
\begin{equation}\label{sigma11--2}
\d \Sigma_2(t)=\d  -\int_0^t  \frac{m(s)}{(1+s)^2}F_1(s) \,ds= -\int_0^t  (1+s)^{\mu-2}F_1(s) \,ds,
\end{equation}
and
\begin{equation}\label{sigma21}
\Sigma_3(t)=-\frac{m(t)}{(1+t)^2}F_1(t)=-  (1+t)^{\mu-2}F_1(t).
\end{equation}
Thanks to \eqref{est-G1} and the fact that $G_1(t)=e^{t} \rho(t)F_1(t)$, we deduce that $\Sigma_1(t) \ge 0$.\\
From \eqref{def231}, we obtain
\begin{equation}\label{def2313}\d F_1(t)= F_1(0)e^{-t} +e^{-t} \int_0^t e^{s}F_2(s)ds,\end{equation}
that we plug in   \eqref{sigma11--2} and we integrate by parts, we deduce that
\begin{eqnarray}
&& \int_0^t  (1+s)^{\mu-2}F_1(s) \,ds= F_1(0)\int_0^t (1+s)^{\mu-2}e^{-s} ds\\
&& + \left(\int_0^t (1+s)^{\mu-2}e^{-s}ds\right) \left(\int_0^t e^{s}F_2(s)  ds\right)- \int_0^t e^{s}F_2(s)\left(\int_0^s(1+\tau)^{\mu-2}e^{-\tau} d\tau\right) ds.
\nonumber
\end{eqnarray}
Hence, we infer that
\begin{equation}\label{sigma111}
\big| \int_0^t (1+s)^{\mu-2}F_1(s) ds \big| \le C F_1(0) + C \int_0^t e^{s}|F_2(s)| ds.
\end{equation}
Therefore we have
\begin{equation}\label{sigma111-1}
|\Sigma_2(t)| \le C F_1(0) + C \int_0^t e^{s}|F_2(s)| ds.
\end{equation}
Using  \eqref{def2313} and similar estimates as for $\Sigma_2(t)$, we easily conclude that
\begin{equation}\label{sigma1112}
|\Sigma_3(t)| \le C F_1(0) + C \int_0^t |F_2(s)| ds.
\end{equation}
Employing \eqref{F1def1}, we recall here that $F_1(0)=\ep \int_{\R^N}f(x)\phi(x)dx$.\\
Combining \eqref{sigma111-1} and \eqref{sigma1112} in \eqref{F1+bis3} and using $m(t) \ge 1$, we obtain
\begin{align}\label{F1+bis45}
\begin{array}{rcl}
\d \frac{d}{dt} \left\{F_2(t)m(t)\right\}+   2m(t)F_2(t)
&\ge& \d  \int_0^t\int_{\R^N}|u_t(x,s)|^p\psi_0(x,s)dx \, ds\vspace{.2cm}\\
\d &&\d -C_0 \ep_0\nu^2  - C_0 \nu^2 \int_0^t e^{s}|F_2(s)| ds,
\end{array}
\end{align}
where $C_0=C_0(\mu, f, N)$.\\
Using the definition of $F_2(t)$, given by \eqref{F2def1}, and Lemma \ref{lem1}, we have
\begin{equation}\label{f2-less}
\begin{array}{rcl}
\d C_0\nu^2 e^{t}|F_2(t)| &\le&\d  \int_{\R^N}|u_t(x,t)|^p\psi_0(x,t)dx + C\nu^{\frac{2p}{p-1}}e^{\frac{p}{p-1}t}\int_{|x|\leq t+R}\psi_0(x,t)dx \vspace{.2cm}\\
&\le& \d \int_{\R^N}|u_t(x,t)|^p\psi_0(x,t)dx + C \nu^{\frac{2p}{p-1}} e^{\frac{p}{p-1}t}(1+t)^{\frac{N-1}{2}}.
\end{array}
\end{equation}
Integrating \eqref{f2-less} in time yields
\begin{equation}\label{f2-less-1}
C_0 \nu^2\int_0^t e^{s}|F_2(s)| ds \le  \int_0^t\int_{\R^N}|u_t(x,s)|^p\psi_0(x,s)dx \, ds + C \nu^{\frac{2p}{p-1}} e^{\frac{p}{p-1}t} (1+t)^{\frac{N-1}{2}}.
\end{equation}
From \eqref{F1+bis45} and \eqref{f2-less-1} we infer that
\begin{align}\label{F1+bis5}
\begin{array}{rcl}
\d \frac{d}{dt} \left\{F_2(t)m(t)\right\}+   2m(t)F_2(t)
+C \nu^2+ C \nu^{\frac{2p}{p-1}}e^{\frac{p}{p-1}t} (1+t)^{\frac{N-1}{2}} \ge 0,
\end{array}
\end{align}
which can be written as
\begin{align}\label{F1+bis4}
\frac{d}{dt} \left\{e^{2t}F_2(t)m(t)\right\}+C \nu^2e^{2t}+C \nu^{\frac{2p}{p-1}}e^{\frac{3p-2}{p-1}t} (1+t)^{\frac{N-1}{2}}\ge 0.
\end{align}
Integrating the above inequality in time gives
\begin{align}\label{F1+bis5}
F_2(t) +C \nu^2 \frac{e^{-2t}}{m(t)}\int_0^t e^{2s}  ds +C\nu^{\frac{2p}{p-1}} \frac{e^{-2t}}{m(t)} \int_0^t e^{\frac{3p-2}{p-1}s} (1+s)^{\frac{N-1}{2}} ds \ge\frac{e^{-2t}}{m(t)}F_2(0) \ge 0.
\end{align}
Hence, we deduce that
\begin{align}\label{F1+bis6}
F_2(t) +C \nu^2 (1+t)^{-\mu}+C \nu^{\frac{2p}{p-1}}e^{\frac{p}{p-1}t} (1+t)^{\frac{N-1}{2}-\mu} \ge 0.
\end{align}
Recall that $G_2(t)=e^{t} \rho(t)F_2(t)$, we obtain
\begin{align}\label{G2+bis61}
G_2(t)+C \nu^2 e^{t} \rho(t)(1+t)^{-\mu}+C \nu^{\frac{2p}{p-1}}e^{t} \rho(t)e^{\frac{p}{p-1}t} (1+t)^{\frac{N-1}{2}-\mu} \ge  0.
\end{align}
On the other hand, using \eqref{lmabdaK} and \eqref{est-double}, we get
 \begin{equation}\label{pho-est-bis}
 \d \rho(t)e^{t} \le C (1+t)^{\frac{\mu}{2}}, \ \forall \ t \ge 0.
 \end{equation}
Finally, from \eqref{G2+bis61} and \eqref{pho-est-bis}, we conclude \eqref{G2+bis7}.

This ends the proof of Lemma \ref{F1-2}.
\end{proof}

\begin{rem}
We note that $G_1(t)$ is positive for all $t \in (0,T)$ thanks to \eqref{est-G1}, and accordingly the same holds for $F_1(t)$.  However, $G_2(t)$ may not be positive all the time and so is for $F_2(t)$; see  the figures in the Appendix. 
 In fact, the functional  $G_2(t)$ may start with negative values for small times.
\end{rem}

\begin{rem}
The estimate \eqref{G2+bis7}, obtained in Lemma \ref{F1-2} for $G_2(t)$, constitutes a first observation useful in obtaining later on the lower bound for $G_2(t)$ for $t$ large enough. In fact, the negative bound in  \eqref{G2+bis7} is due to the presence of a mass term in \eqref{T-sys}. Obviously, for $\nu =0$, we find again here the known result on the positivity of $G_2(t)$ in the absence of the mass term, see e.g. \cite{Our2}.
\end{rem}

We will see in the following that the functional $G_2(t)$, after taking some negative values for small time, becomes positive for large time.  The last assertion is obtained in Lemma \ref{F11} below thanks to to the compensation of the negative sign of the linear part in the functional $G_2(t)$ by the time derivative nonlinearity. However, the nonlinearity $|u|^q$ is not involved in the proofs of Lemmas \ref{F1-2} and \ref{F11}. This allows us to use the result in Lemma \ref{F11} for the problem \eqref{T-sys-bis} to prove Theorem \ref{th_u_t} in Section \ref{sec-ut} below.

Now we are in a position to prove  the following lemma.
\begin{lem}\label{F11}
For any energy solution $u$ of the system \eqref{T-sys} with initial data satisfying the assumptions in Theorem \ref{blowup},  there exists $T_1>0$ such that 
\begin{equation}
\label{F2postive}
G_2(t)\ge C_{G_2}\, \e, 
\quad\text{for all}\ t  \ge  T_1=-\ln(\e),
\end{equation}
where $C_{G_2}$ is a positive constant which depends on $p,f$, $g$, $N,R,\ep_0,\nu$ and $\mu$.
\end{lem}
 
\begin{proof}
Let $t \in [0,T)$. Using \eqref{test11}, \eqref{F1def}, \eqref{F2def} and the fact that
 \begin{equation}\label{def23}\d G_1'(t) -\frac{\rho'(t)}{\rho(t)}G_1(t)= G_2(t),\end{equation}
 the equation  \eqref{eq6} implies
\begin{equation}
\begin{array}{l}\label{eq5bis}
\d G_2(t)+\left(\frac{\mu}{1+t}-\frac{\rho'(t)}{\rho(t)}\right)G_1(t)\\
=\d \int_0^t\int_{\R^N}\left\{|u_t(x,s)|^p+|u(x,s)|^q\right\}\psi(x,s)dx \, ds +\e \, C(f,g).
\end{array}
\end{equation}
Differentiating in time  \eqref{eq5bis} yields
\begin{align}\label{F1+bis}
\d G_2'(t)+\left(\frac{\mu}{1+t}-\frac{\rho'(t)}{\rho(t)}\right)G'_1(t)-\left(\frac{\mu}{(1+t)^2}+\frac{\rho''(t)\rho(t)-(\rho'(t))^2}{\rho^2(t)}\right)G_1(t) \vspace{.2cm}\\
=\int_{\R^N}\left\{|u_t(x,t)|^p+|u(x,t)|^q\right\}\psi(x,t)dx.\nonumber
\end{align}
Exploiting  \eqref{lambda} and   \eqref{def23}, the equation \eqref{F1+bis} can be written as follows:
\begin{align}\label{F1+bis2}
\d G_2'(t)+\left(\frac{\mu}{1+t}-\frac{\rho'(t)}{\rho(t)}\right)G_2(t)+\left(-1+\frac{\nu^2}{(1+t)^2}\right)G_1(t)\\ =\int_{\R^N}\left\{|u_t(x,t)|^p+|u(x,t)|^q\right\}\psi(x,t)dx. \nonumber
\end{align}
Thanks to the definition of $\Gamma(t)$ given by \eqref{gamma}, we infer that 
\begin{equation}\label{G2+bis3}
\begin{array}{c}
\d G_2'(t)+\frac{3\Gamma(t)}{4}G_2(t)\ge\Sigma_4(t)+\Sigma_5(t)+\int_{\R^N}\left\{|u_t(x,t)|^p+|u(x,t)|^q\right\}\psi(x,t)dx,
\end{array}
\end{equation}
where 
\begin{equation}\label{sigma1-exp}
\Sigma_4(t):=\d \left(-\frac{\rho'(t)}{2\rho(t)}-\frac{\mu}{4(1+t)}\right)\left(G_2(t)+\left(\frac{\mu}{1+t}-\frac{\rho'(t)}{\rho(t)}\right)G_1(t)\right),
\end{equation}
and
\begin{equation}\label{sigma2-exp}
\Sigma_5(t):=\d \left(1-\frac{\nu^2}{(1+t)^2}+\left(\frac{\rho'(t)}{2\rho(t)}+\frac{\mu}{4(1+t)}\right) \left(\frac{\mu}{1+t}-\frac{\rho'(t)}{\rho(t)}\right) \right)  G_1(t).
\end{equation}
Making use of \eqref{eq5bis} and \eqref{lambda'lambda1}, we have the existence of $\tilde{T}_1=\tilde{T}_1(\mu, \nu) \ge T_0$ such that
\begin{equation}\label{sigma1}
\d \Sigma_4(t) \ge C \, \e + \frac14 \int_0^t\int_{\R^N}\left\{|u_t(x,s)|^p+|u(x,s)|^q\right\}\psi(x,s)dx \, ds, \quad \forall \ t \ge \tilde{T}_1. 
\end{equation}
Now, using Lemma \ref{F1} and \eqref{lambda'lambda1}, we deduce that there exists $\tilde{T}_2=\tilde{T}_2(\mu, \nu) \ge \tilde{T}_1(\mu, \nu)$ verifying
\begin{equation}\label{sigma2}
\d \Sigma_5(t) \ge 0, \quad \forall \ t  \ge  \tilde{T}_2. 
\end{equation}
Gathering \eqref{G2+bis3}, \eqref{sigma1} and \eqref{sigma2}, we get
\begin{equation}\label{G2+bis4}
\begin{array}{l}
\d G_2'(t)+\frac{3\Gamma(t)}{4}G_2(t)\ge C \, \e+\int_{\R^N}\left\{|u_t(x,t)|^p+|u(x,t)|^q\right\}\psi(x,t)dx \vspace{.2cm}\\
\d + \frac14 \int_0^t\int_{\R^N}\left\{|u_t(x,s)|^p+|u(x,s)|^q\right\}\psi(x,s)dx \, ds, \quad \forall \ t  \ge  \tilde{T}_2.
\end{array}
\end{equation}
At this level we can ignore the nonlinear terms. In fact, we could remove the nonlinear terms from almost the beginning of the proof (say \eqref{F1+bis} for example), but we adopted to keep the nonlinear terms in \eqref{G2+bis4} to make it useful in the proof of Theorem  \ref{th_u_t} in Section \ref{sec-ut} below. 
Hence, we have
\begin{equation}\label{G2+bis41}
\begin{array}{l}
\d G_2'(t)+\frac{3\Gamma(t)}{4}G_2(t)\ge C \, \e, \quad \forall \ t  \ge  \tilde{T}_2.
\end{array}
\end{equation}
Multiplying  \eqref{G2+bis41} by $\frac{(1+t)^{3\mu/4}}{\rho^{3/2}(t)}$ and integrating over $(\tilde{T}_2,t)$, we infer  that
\begin{align}\label{est-G111-bis}
 G_2(t)
\ge G_2(\tilde{T}_2)\frac{\rho^{3/2}(t)}{(1+t)^{3\mu/4}}+C\,{\e}\frac{\rho^{3/2}(t)}{(1+t)^{3\mu/4}}\int_{\tilde{T}_2}^t\frac{(1+s)^{3\mu/4}}{\rho^{3/2}(s)}ds, \quad \forall \ t  \ge  \tilde{T}_2.
\end{align}
Thanks to \eqref{G2+bis7} we have
\begin{equation}\label{G2sup}
G_2(\tilde{T}_2) \ge - \tilde{\mathcal{K}},
\end{equation}
where $\tilde{\mathcal{K}}:=
\mathcal{K}\nu^2\left\{ 1 +\nu^{\frac{2}{p-1}}e^{\frac{p}{p-1}\tilde{T}_2} (1+\tilde{T}_2)^{\frac{N-1}{2}} \right\}$.\\
Recalling \eqref{est-double} and \eqref{G2sup}, we deduce from \eqref{est-G111-bis} that for all $t \ge \tilde{T}=\tilde{T}(\mu,\nu):=2\tilde{T}_2$, we have
\begin{align}\label{est-G2-12}
 G_2(t)
&\ge    -\tilde{\mathcal{K}} e^{-3t/2}+ C\,{\e}e^{-3t/2} \int^t_{t/2}e^{3s/2}ds\\
&\ge  -\tilde{\mathcal{K}} e^{-3t/2}+ C\,{\e},
\end{align}
Therefore, for $\ep$ small, we get
\begin{align}\label{est-G1-2}
 G_2(t)
\ge  C_{G_2}\,{\e}, \quad \forall \ t \ge T_1:=-\ln(\e).
\end{align}

 This concludes the proof of Lemma
\ref{F11}.
\end{proof}

\begin{rem}\label{rem3.1}
Notice that in the proof of Lemma \ref{F1}  we only used the positivity of each one of the nonlinearities ($|u_t|^p$ and $|u|^q$). Indeed, the result in this lemma is based on the comprehension of the dynamics in the linear part and, thus, the same conclusion can be handled similarly for any positive nonlinearity of the form $\mathcal{N}(u,u_t)$ instead of $|u_t|^p+|u|^q$. Furthermore, in the proof of Lemma \ref{F11} we use  the result on the negative lower bound of $G_2(t)$ obtained in Lemma \ref{F1-2} where we make use of the nonlinearity $|u_t|^p$ to control in part the negativity of $G_2(t)$. Although the nonlinear terms could  be ignored from the beginning of the proof of  Lemma \ref{F11}, but, we chose to keep them at certain level throughout the proof for later use in the proof of  Theorem  \ref{th_u_t}.
\end{rem}

\begin{rem}\label{rem3.2}
Naturally,  the  results of Lemmas \ref{F1} and \ref{F11}  hold true when we consider a more general nonlinearity $\mathcal{N}(u,u_t)=|u_t|^p+\tilde{\mathcal{N}}(u,u_t)$ (with $\tilde{\mathcal{N}}(u,u_t) \ge 0$) instead of $|u_t|^p+|u|^q$,  as it is the case for example in (\ref{T-sys-bis}). 
\end{rem}
\section{Proof of Theorem \ref{blowup}}\label{proof}
\par\quad
The aim of this section is to prove the first  theorem in this article, namely Theorem \ref{blowup}, which is related to the blow-up result and the  lifespan estimate of the solution of (\ref{T-sys}). To this end, we will  employ the lemmas proven in Section \ref{aux} and a Kato's lemma type.

First, using the hypotheses in Theorem \ref{blowup}, we recall that $\mbox{\rm supp}(u)\ \subset\{(x,t)\in\R^N\times[0,\infty): |x|\le t+R\}$. \\
Let $t \in [0,T)$. Then, thanks to the hypotheses in Theorem \ref{blowup}, we define
\begin{equation}
F(t):=\int_{\R^N}u(x,t)dx.
\end{equation}
By choosing the test function $\Phi$ in \eqref{energysol2} such that
$\Phi\equiv 1$ in $\{(x,s)\in \R^N\times[0,t]:|x|\le s+R\}$\footnote{ The choice $\Phi\equiv 1$ is possible since the initial data $f$ and $g$ are supported on $B_{\R^N}(0,R)$.} and using the definition of $F(t)$, we obtain
\begin{equation}
\label{F'0ineq12}
F'(t)+ \int_0^t  \frac{\mu}{1+s}F'(s) ds+\int_0^t  \frac{\nu^2}{(1+s)^2}F(s)\,ds= F'(0)+ \int_0^t \int_{\R^N}\left\{|u_t(x,s)|^p+|u(x,s)|^q\right\}dx \,ds.
\end{equation}
Differentiating in time the equation \eqref{F'0ineq12}, we have
\begin{equation}
\label{F'0ineq1-bis}
F''(t)+  \frac{\mu}{1+t}F'(t) +\frac{\nu^2}{(1+t)^2}F(t)= \int_{\R^N}\left\{|u_t(x,t)|^p+|u(x,t)|^q\right\}dx.
\end{equation}
In order to get rid of the mass term in \eqref{F'0ineq1-bis} ({\it i.e.} $\frac{\nu^2}{(1+t)^2}F(t)$), we  introduce a  new functional $G(t)$ which is defined as
\begin{equation}
\label{G-exp}
G(t):=\zeta(t)F(t) \ \text{with} \ \d \zeta(t)=(1+t)^{\al},
\end{equation}
where $\al$ is given by \eqref{alpha}.\\
Using \eqref{G-exp}, the equation \eqref{F'0ineq1-bis} yields
\begin{equation}
\label{G-second}
G''(t)+  \frac{1 +\sqrt{\de}}{1+t}G'(t) = (1+t)^{\al}\int_{\R^N}\left\{|u_t(x,t)|^p+|u(x,t)|^q\right\}dx.
\end{equation}
Now, we introduce the following multiplier
\begin{equation}
\label{test1}
\begin{aligned}
\mathcal{M}(t):=(1+t)^{1 +\sqrt{\de}}. 
\end{aligned}
\end{equation}
Multiplying \eqref{G-second} by $\mathcal{M}(t)$ and integrating over $(0,t)$, we infer that
\begin{equation}
\label{F'0ineq11}
\mathcal{M}(t)G'(t)= G'(0)+ \int_0^t \mathcal{M}(s) (1+s)^{\al}\int_{\R^N}\left\{|u_t(x,s)|^p+|u(x,s)|^q\right\}dx \,ds.
\end{equation}
Observe that $G'(0)=\frac{\mu-1-\sqrt{\de}}{2}\int_{\R^N}f(x)  dx+\int_{\R^N} g(x) dx > 0$ thanks to the hypothesis \eqref{hypfg}. Hence, we have
\begin{equation}\label{F'0ineq11-22}
\mathcal{M}(t)G'(t) \ge  \int_0^t \mathcal{M}(s) (1+s)^{\al}\int_{\R^N}\left\{|u_t(x,s)|^p+|u(x,s)|^q\right\}dx \,ds.
\end{equation}
Integrating \eqref{F'0ineq11-22} over $(0,t)$, after dividing it by $\mathcal{M}(t)$, and using the fact that 
 $G(0)=\int_{\R^N}f(x)  dx \ge 0$, we infer that
\begin{align}
\label{F'0ineqmmint}
G(t)\geq \int_0^t\frac{1}{\mathcal{M}(s)} \int_0^s \mathcal{M}(\tau ) (1+\tau)^{\al}\int_{\R^N}\left\{|u_t(x,\tau)|^p+|u(x,\tau)|^q\right\} dx \,d\tau\,ds.
\end{align}
Utilizing the estimates \eqref{psi} and \eqref{F2postive} together with H\"{o}lder's inequality,  a lower bound fo the nonlinear term can be obtained as follows:
\begin{equation}
\begin{array}{rcl}
\d \int_{\R^N}|u_t(x,t)|^pdx &\geq& \d G_2^p(t)\left(\int_{|x|\leq t+R}\Big(\psi(x,t)\Big)^{\frac{p}{p-1}}dx\right)^{-(p-1)} \vspace{.2cm}\\  &\geq&  C\rho^{-p}(t)e^{-pt}\e^p(1+t)^{ -\frac{(N-1)(p-2)}2}, \quad \forall \ t \ge T_1,
\end{array}
\end{equation}
where $T_1$ is defined by  \eqref{F2postive}.\\
From \eqref{lmabdaK} and \eqref{est-double}, we deduce that
 \begin{equation}\label{pho-est}
 \d \rho(t)e^{t} \le C (1+t)^{\frac{\mu}{2}}, \ \forall \ t \ge T_0/2 \quad (T_0 < T_1).
 \end{equation}
Hence, we get
\begin{equation}
\d \int_{\R^N}|u_t(x,t)|^pdx \geq C \e^p(1+t)^{ -\frac{\mu p+(N-1)(p-2)}2}, \ \forall \ t \ge T_1.\\
\end{equation}
Combining the above inequality  with \eqref{F'0ineqmmint} yields
\begin{equation}
\begin{aligned}\label{F0first}
G(t)
&\geq C\e^p (1+t)^{2+\al -\frac{\mu p+(N-1)(p-2)}2}, \ \forall \ t \ge T_1.
\end{aligned}
\end{equation}
Again here thanks to the fact that $\mbox{\rm supp}(u)\ \subset\{(x,t)\in\R^N\times[0,\infty): |x|\le t+R\}$, we have
\begin{align}
\Big(\int_{\R^N}u(x,t)dx\Big)^q\le C  \big(t+1 \big)^{N(q-1)}\int_{|x|\le t+R}|u(x,t)|^qdx, 
\end{align}
and, hence, we deduce that
\begin{align}\label{f0qsup}
G^q(t)\le  C\big(t+1 \big)^{N(q-1)+\al q}  \int_{|x|\le t+R}|u(x,t)|^q dx.
\end{align}
Differentiating in time \eqref{F'0ineq11}, we obtain 
\begin{equation}
\label{F'0ineq1}
(\mathcal{M}(t)G'(t))'=  \mathcal{M}(t) (1+t)^{\al}\int_{\R^N}\left\{|u_t(x,t)|^p+|u(x,t)|^q\right\}dx \ge  \mathcal{M}(t) (1+t)^{\al}\int_{\R^N}|u(x,t)|^q dx.
\end{equation}
 Incorporating  \eqref{f0qsup} into \eqref{F'0ineq1} and dividing by $\mathcal{M}(t)$ the new equation resulting from \eqref{F'0ineq1}), we get for $L(t):=\sqrt{\mathcal{M}(t)}G(t)$,
\begin{equation}
\label{F'0ineq2}
L''(t) +\frac{1 -\de}{4(1+t)^2} L(t)\ge  C\frac{L^q(t)}{\big(1+t \big)^{(N+\frac{\mu}{2})(q-1)}}, \ \forall \ t > 0.
\end{equation}

At this level, we recall that $L(t) \ge 0$ thanks to the positivity of $G(t)$ which is obtained in \eqref{F'0ineqmmint}. Therefore, two cases will presented in the subsequent depending on the value of the parameter $\de$, defined in \eqref{delta}. \\

{\bf First case ($\de \ge 1$).} 

Since  $L(t)$ is nonnegative, the estimate \eqref{F'0ineq2} yields
\begin{equation}
\label{F'0ineq2-1}
L''(t) \ge  C\frac{L^q(t)}{\big(1+t \big)^{(N+\frac{\mu}{2})(q-1)}}, \ \forall \ t > 0.
\end{equation}
Recall the definition of $L(t):=\sqrt{\mathcal{M}(t)}G(t)$ and using  \eqref{F'0ineq11-22} and \eqref{F'0ineqmmint}, we deduce that $L'(t) \ge 0$. Hence, multiplying \eqref{F'0ineq2-1} by $L'(t)$ gives 
\begin{equation}
\label{F25nov3}
\left\{\Big(L'(t)\Big)^2\right\}'
\ge C\frac{\Big(L^{q+1}(t)\Big)'}{(1+t)^{(N+\frac{\mu}{2})(q-1)}}, \ \forall \ t > 0.
\end{equation}
A simple integration in time of \eqref{F25nov3} yields
\begin{equation}
\label{F25nov4-1}
\Big(L'(t)\Big)^2
\ge C\frac{L^{q+1}(t)}{(1+t)^{(N+\frac{\mu}{2})(q-1)}} +\left((L'(0))^2-CL^{q+1}(0)\right), \ \forall \ t > 0.
\end{equation}
For $\e$ small enough, thanks to the hypothesis on the smallness of the initial data, we obviously have the positivity of the last term in the right-hand side of \eqref{F25nov4-1}.\\
Therefore, the estimate \eqref{F25nov4-1} implies that
\begin{equation}
\label{F25nov6}
\frac{L'(t)}{L^{1+\theta}(t)}
\ge C \frac{L^{\frac{q-1}2-\theta}(t)}{(1+t)^{\frac{(2N+\mu)(q-1)}{4}}}, \ \forall \ t > 0,
\end{equation}
for $\theta>0$ small enough.\\

{\bf Second case ($\de < 1$).}

First, we recall that $L'(t)>0$. Then, multiplying \eqref{F'0ineq2} by $(1+t)^2 L'(t)$ yields
\begin{align}
\label{F25nov3}
\frac{(1+t)^2}{2}\left(\left(L'(t)\right)^2\right)' + \frac{1 -\de}{8} \left(L^2(t)\right)'\\
\ge C\frac{\Big(L^{q+1}(t)\Big)'}{\big(1+t \big)^{(N+\frac{\mu}{2})(q-1)-2}}, \ \forall \ t > 0.\nonumber
\end{align}
We integrate the above inequality and observe that $t \mapsto 1/\big(1+t \big)^{(N+\frac{\mu}{2})(q-1)-2}$ is a decreasing function (thanks to $\d N(q-1)-2>0$ since   $q>1+\frac{2}{N}$ which is related to the case $q>q_S(N+\mu)$\footnote{  Obviously if $q \le q_S(N+\mu)$ the blow-up result can be proven by only considering the nonlinearity $|u(x,s)|^q$.}). Hence, we obtain
\begin{equation}
\label{F25nov4}
\begin{array}{c}
\d \frac{(1+t)^2}{2}\left(L'(t)\right)^2 + \frac{1 -\de}{8} L^2(t)
\ge C_1\d \frac{L^{q+1}(t)}{\big(1+t \big)^{(N+\frac{\mu}{2})(q-1)-2}}\vspace{.2cm} \\+\d L^2(0)\left(\frac{1 -\de}{8}- C L^{q-1}(0)\right), \ \forall \ t > 0.
\end{array}
\end{equation}
Again here, we simply show that the last term in the right-hand side of \eqref{F25nov4} is positive using the smallness of the initial data ($\e$ small enough). Therefore we infer that
\begin{equation}
\label{F25nov5}
\d \frac{(1+t)^2}{2}\left(L'(t)\right)^2 + \frac{1 -\de}{8} L^2(t)
\ge C_1\d \frac{L^{q+1}(t)}{\big(1+t \big)^{(N+\frac{\mu}{2})(q-1)-2}}.
\end{equation}
Utilizing  the estimate  \eqref{F0first}, the expression of $L(t)$, the definition of $\lambda(p, q, N)$, as in \eqref{1.5}, and the expression of $\mathcal{M}(t)$ (given by \eqref{test1}), we conclude that
\begin{equation}
\label{Gq-11}
\frac{L^{q-1}(t)}{\big(1+t \big)^{(N+\frac{\mu}{2})(q-1)-2}}>C_2  \e^{p(q-1)}(1+t)^{2-\frac{\lambda(p, q, N+\mu)}{2}}, \ \forall \ t \ge T_1(\ep).
\end{equation}
Now, we choose $T_2$ such that
\begin{equation}\label{T2-choice}
T_2= \max \left(C_3^{-\frac{2}{4-\lambda(p, q, N+\mu)}} \e^{-\frac{2p(q-1)}{4-\lambda(p, q, N+\mu)}},T_1(\ep)\right),
\end{equation}
where $\d C_3=4 C_1C_2/ (1 -\de)$ and $T_1(\ep)$ is defined by  \eqref{F2postive}. Note that for $\e$ small enough 
\begin{equation}\label{T2-choice-bis}
T_2=T_2(\e):=C_3^{-\frac{2}{4-\lambda(p, q, N+\mu)}} \e^{-\frac{2p(q-1)}{4-\lambda(p, q, N+\mu)}}.\end{equation}
Hence, the above choice of $T_2$ implies that
\begin{equation}
\label{F25nov6}
\frac{L^{q-1}(t)}{\big(1+t \big)^{(N+\frac{\mu}{2})(q-1)-2}}>\frac{1 -\de}{4C_1}, \ \forall \ t\ge T_2,
\end{equation} 
Now, combining \eqref{F25nov6} in \eqref{F25nov5}, we obtain the following estimate:
\begin{equation}
\label{F25nov7}
\d (1+t)^2\left(L'(t)\right)^2 
\ge C_1\d \frac{L^{q+1}(t)}{\big(1+t \big)^{(N+\frac{\mu}{2})(q-1)-2}}, \ \forall \ t\ge T_2,
\end{equation}
that we rewrite as
\begin{equation}
\label{F25nov8}
\frac{L'(t)}{L^{1+\theta}(t)}
\ge C  \frac{L^{\frac{q-1}2-\theta}(t)}{(1+t)^{\frac{(2N+\mu)(q-1)}4}}, \ \forall \ t\ge T_2,
\end{equation}
for $\theta>0$ small enough. \\

Finally, for  $\de \ge 1$ or $\de < 1$, we obtain almost the same estimates \eqref{F25nov6} and \eqref{F25nov8}, respectively, however, they only differ by the starting times which are $0$ and $T_2$, respectively. In conclusion, the  estimate \eqref{F25nov8} is true in both cases for all $t\ge T_2$ where $T_2$ is given by \eqref{T2-choice-bis}.

The rest of the proof follows the same lines as in the corresponding part in the proof  of Theorem 2.2 in \cite[Section 4]{Our2}  which starts from (4.30) in the same paper \cite{Our2}.

This achieves the proof of Theorem \ref{blowup}.\hfill $\Box$

\section{Proof of Theorem \ref{th_u_t}.}\label{sec-ut}

We are interested in this  section in proving Theorem \ref{th_u_t} which is  related to the derivation  of the critical exponent associated with the nonlinear term in the problem \eqref{T-sys-bis}. As mentioned earlier in this work, we will make use of the computations already done in Section \ref{aux}. More precisely, we recall that Lemma \ref{F1}  remains true for the solution of \eqref{T-sys-bis} (see Remark \ref{rem3.1}) since we only use the fact that the nonlinear terms are positive. Furthermore, Lemma \ref{F11}, which is based on the result of Lemma \ref{F1-2}, only uses the nonlinear time derivative term $|u_t|^p$ and therefore remains true for the solution of \eqref{T-sys-bis}.

In fact, we proved in Lemma \ref{F11}  that $G_2(t)$ is  coercive  starting from relatively large time which is increasing as the initial data are getting smaller, namely as $\ep \to 0$. This  observation constitutes a novelty for \eqref{T-sys-bis} compared to the equation without mass; see e.g. \cite{Our2}.

Taking advantage from the above observation about $G_2(t)$, we improve the blow-up result in \cite{Palmieri} for   $p \in (1, p_G(N+\sigma)]$, where $p_G(N)$ is the Glassey exponent given by \eqref{Glassey} and $\sigma$ is given by \eqref{sigma}, to reach the new blow-up region $p \in (1, p_G(N+\mu))$. Indeed, our result for (\ref{T-sys-bis}) enhances the corresponding one in  \cite{Palmieri}, for  $\de < 1$, and coincides with it for $\de \ge 1$. In particular, we may conjecture that the mass term {\it has no influence} on the dynamics for $\de \ge 0$, {\it i.e.}, $\nu^2 \le \frac{(\mu -1)^2}{4}$, by simply comparing \cite[Theorem 2.4]{Our2} and Theorem \ref{th_u_t} in the present work. Finally, we believe that the derived limiting exponent $p_G(N+\mu)$ may get to the threshold between the blow-up and the global existence regions.\\

In the subsequent we will use the  estimate \eqref{G2+bis4} with omitting the nonlinear term $|u(x,t)|^q$ and keeping the other  nonlinearity $|u_t(x,t)|^p$. Hence, we obtain
\begin{equation}\label{G2+bis5}
\begin{array}{rcl}
\d G_2'(t)+\frac{3\Gamma(t)}{4}G_2(t)&\ge& \d \frac{1}{4}\int_{0}^t \int_{\R^N}|u_t(x,s)|^p\psi(x,s)dx ds\vspace{.2cm}\\ &+&  \d \int_{\R^N}|u_t(x,t)|^p\psi(x,t) dx+C_5 \, \e, \quad \forall \ t  \ge \tilde{T}_2.
\end{array}
\end{equation}
Let
\[
H(t):=
\frac{1}{8}\int_{T_3(\ep)}^t \int_{\R^N}|u_t(x,s)|^p\psi(x,s)dx ds
+\frac{C_6 \e}{8},
\]
where $T_3(\ep):=\max(T_1, \tilde{T}_2,\tilde{T}_3)$, $C_6=\min(C_5,8C_{G_2})$ ($C_{G_2}$ is defined in Lemma \ref{F11}) and $\tilde{T}_3$ is chosen such that $\frac{1}{4}-\frac{3\Gamma(t)}{32}>0$ and $\Gamma(t)>0$ for all $t \ge\tilde{T}_3$ (this is possible thanks to \eqref{gamma} and \eqref{lambda'lambda1}). Since $T_1$, given by \eqref{F2postive}, is large for $\ep$ small, we can hereafter  set $T_3(\ep)=-\ln(\ep)$. Now, we introduce
$$\mathcal{F}(t):=G_2(t)-H(t),$$
which satisfies
\begin{equation}\label{G2+bis6}
\begin{array}{rcl}
\d \mathcal{F}'(t)+\frac{3\Gamma(t)}{4}\mathcal{F}(t) &\ge& \d \left(\frac{1}{4}-\frac{3\Gamma(t)}{32}\right)\int_{T_3(\ep)}^t \int_{\R^N}|u_t(x,s)|^p\psi(x,s)dx ds\vspace{.2cm}\\ &+&  \d \frac{7}{8}\int_{\R^N}|u_t(x,t)|^p\psi(x,t) dx+C_6 \left(1-\frac{3\Gamma(t)}{32}\right) \e\\
&\ge&0, \quad \forall \ t \ge T_3(\ep).
\end{array}
\end{equation}
Then, the estimate   \eqref{G2+bis6} yields
\begin{align}\label{est-G111}
 \mathcal{F}(t)
\ge \mathcal{F}(T_3(\ep))\frac{(1+T_3(\ep))^{3\mu/4}}{\rho^{3/2}(T_3(\ep))}\frac{\rho^{3/2}(t)}{(1+t)^{3\mu/4}}, \ \forall \ t \ge T_3(\ep),
\end{align}
where $\rho(t)$ is defined by \eqref{lmabdaK}.\\
Hence, we have $\d \mathcal{F}(T_3(\ep))=G_2(T_3(\ep))-\frac{C_6 \e}{8} \ge G_2(T_3(\ep))-C_{G_2}\e \ge 0$ thanks to Lemma \ref{F11} and the fact that $C_6=\min(C_5,8C_{G_2}) \le 8C_{G_2}$. \\
Consequently, we have 
\begin{equation}
\label{G2-est}
G_2(t)\geq H(t), \ \forall \ t \ge T_3(\ep).
\end{equation}
Using the H\"{o}lder's inequality and the estimates \eqref{psi} and \eqref{F2postive}, we can easily see that
\begin{equation}
\begin{array}{rcl}
\d \int_{\R^N}|u_t(x,t)|^p\psi(x,t)dx &\geq&\d G_2^p(t)\left(\int_{|x|\leq t+R}\psi(x,t)dx\right)^{-(p-1)} \vspace{.2cm}\\ &\geq& C G_2^p(t) \rho^{-(p-1)}(t)e^{-(p-1)t}(1+t)^{-\frac{(N-1)(p-1)}2}.
\end{array}
\end{equation}
Thanks to \eqref{pho-est}, we get 
\begin{equation}
\d \int_{\R^N}|u_t(x,t)|^p\psi(x,t)dx \geq C G_2^p(t)(1+t)^{-\frac{(N+\mu-1)(p-1)}2}, \ \forall \ t \ge T_3(\ep).
\end{equation}
From the above estimate and \eqref{G2-est}, we infer that
\begin{equation}
\label{inequalityfornonlinearin}
H'(t)\geq C H^p(t)(1+t)^{-\frac{(N+\mu-1)(p-1)}2}, \quad \forall \ t \ge T_3(\ep).
\end{equation}
Observing that $H(T_3(\ep))=C_6 \e/8>0$, 
we deduce the upper bound of the lifespan estimate as stated in Theorem \ref{th_u_t}.

\section{Appendix}\label{appendix}
In this Appendix we will display some figures obtained by simple computations on Matlab. Indeed, the aim here is to enhance the observations obtained in Lemmas \ref{F1-2} and \ref{F11}, and more precisely to show the behavior of the functional $F_2(t)$, defined by \eqref{F2def1}, for different values of $\de=(\mu-1)^2-4\nu^2$ (and consequently this yields the dynamics of $G_2(t)$). We recall here that 
\begin{displaymath}\label{def231}\d F_2(t)=F_1'(t) +F_1(t),\end{displaymath}
where $F_1(t)$ satisfies the  equation \eqref{F1+bis1} with ignoring the nonlinear terms and using the above equation:
\begin{equation}\label{eq-F2}
\displaystyle F''_1(t)+\left(2+\frac{\mu}{1+t}\right) F'_1(t) +\left(\frac{\mu}{1+t}+\frac{\nu^2}{(1+t)^2}\right) F_1(t) = 0.
\end{equation}
The numerical treatment of \eqref{eq-F2} yields the graphs for $F_2(t)$  as shown below.

\begin{figure}[ht] 
  \label{fig1} 
  \begin{minipage}[b]{0.5\linewidth}
    \centering
    \includegraphics[width=.5\linewidth]{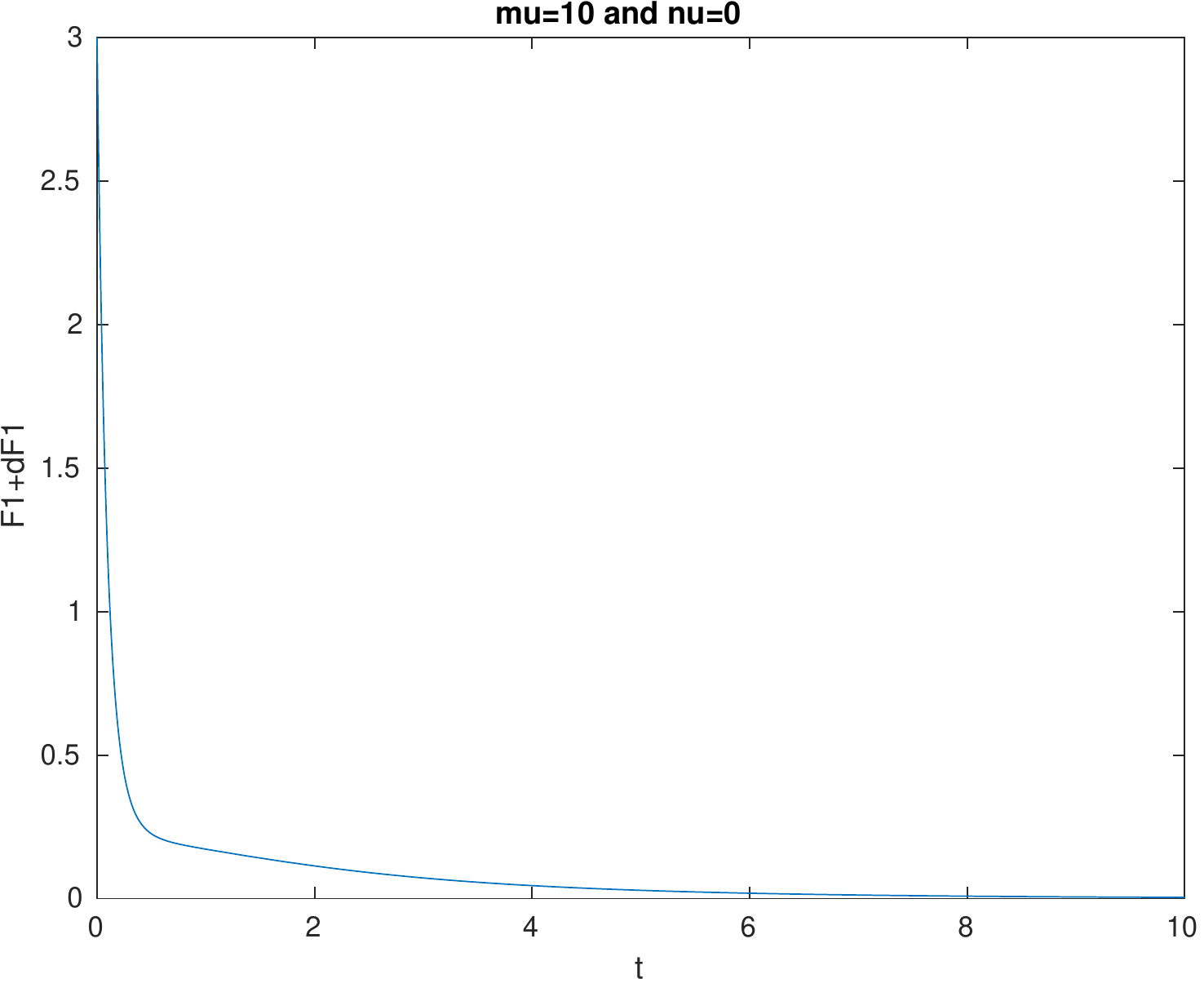} 
    \caption{\tiny{The case $\mu=10, \nu=0$ (the free-mass case with $\de >0$).}} 
    \vspace{4ex}
  \end{minipage}
  \begin{minipage}[b]{0.5\linewidth}
    \centering
    \includegraphics[width=.5\linewidth]{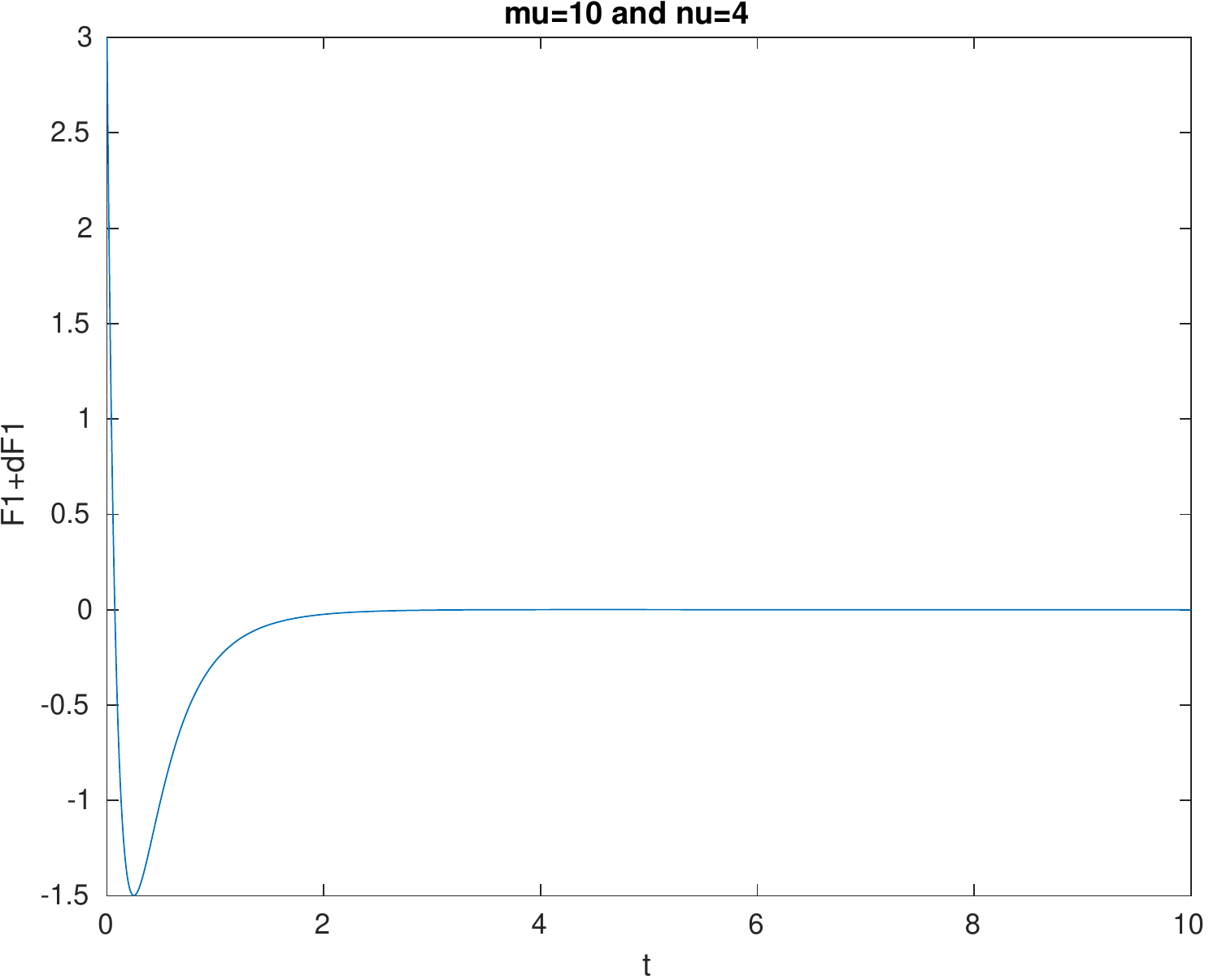} 
    \caption{\tiny{The case $\mu=10, \nu=4$ which corresponds to $\de >0$.}} 
    \vspace{4ex}
  \end{minipage} 
  \begin{minipage}[b]{0.5\linewidth}
    \centering
    \includegraphics[width=.5\linewidth]{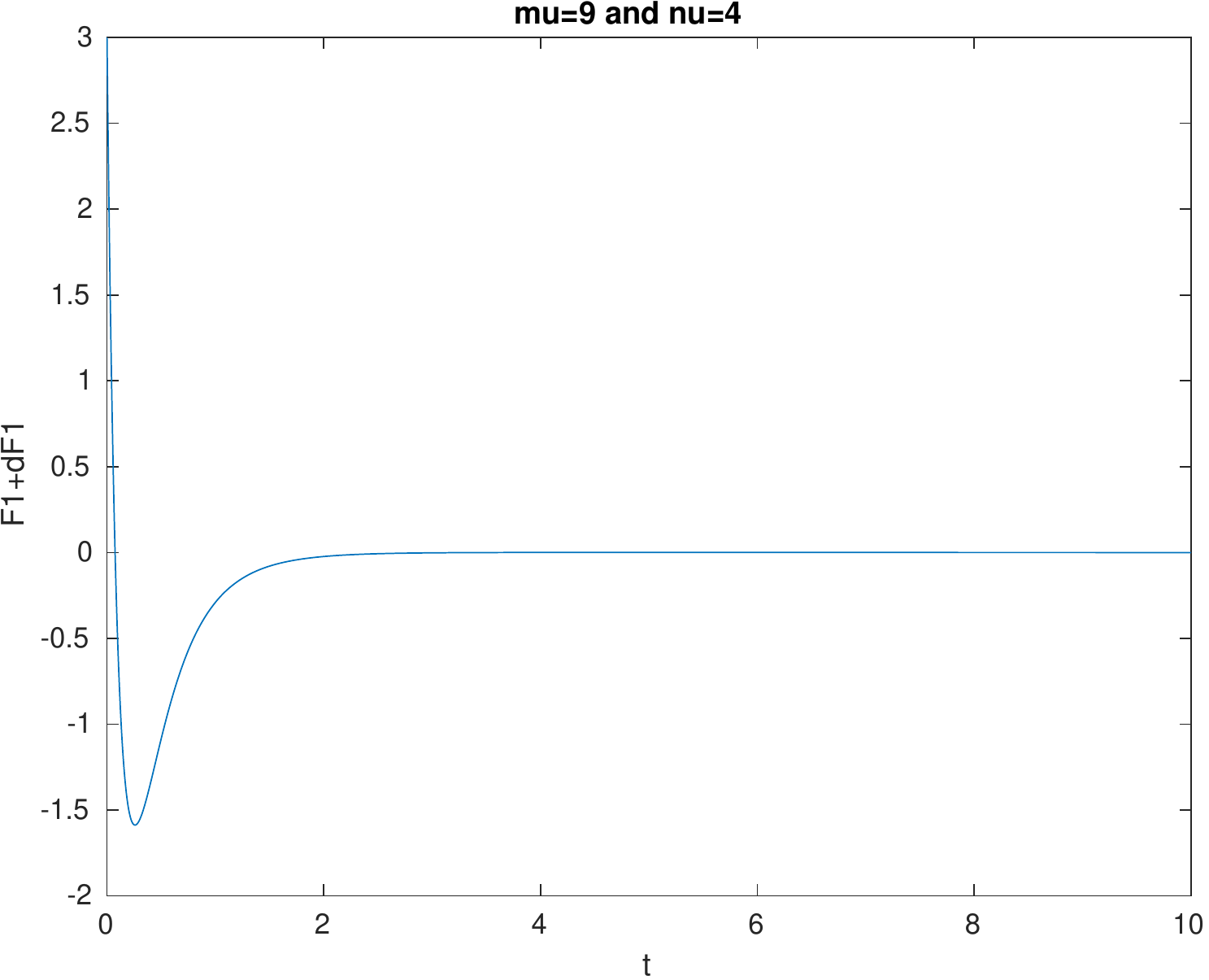} 
    \caption{\tiny{The case $\mu=9, \nu=4$ which corresponds to $\de =0$.}} 
    \vspace{4ex}
  \end{minipage}
  \begin{minipage}[b]{0.5\linewidth}
    \centering
    \includegraphics[width=.5\linewidth]{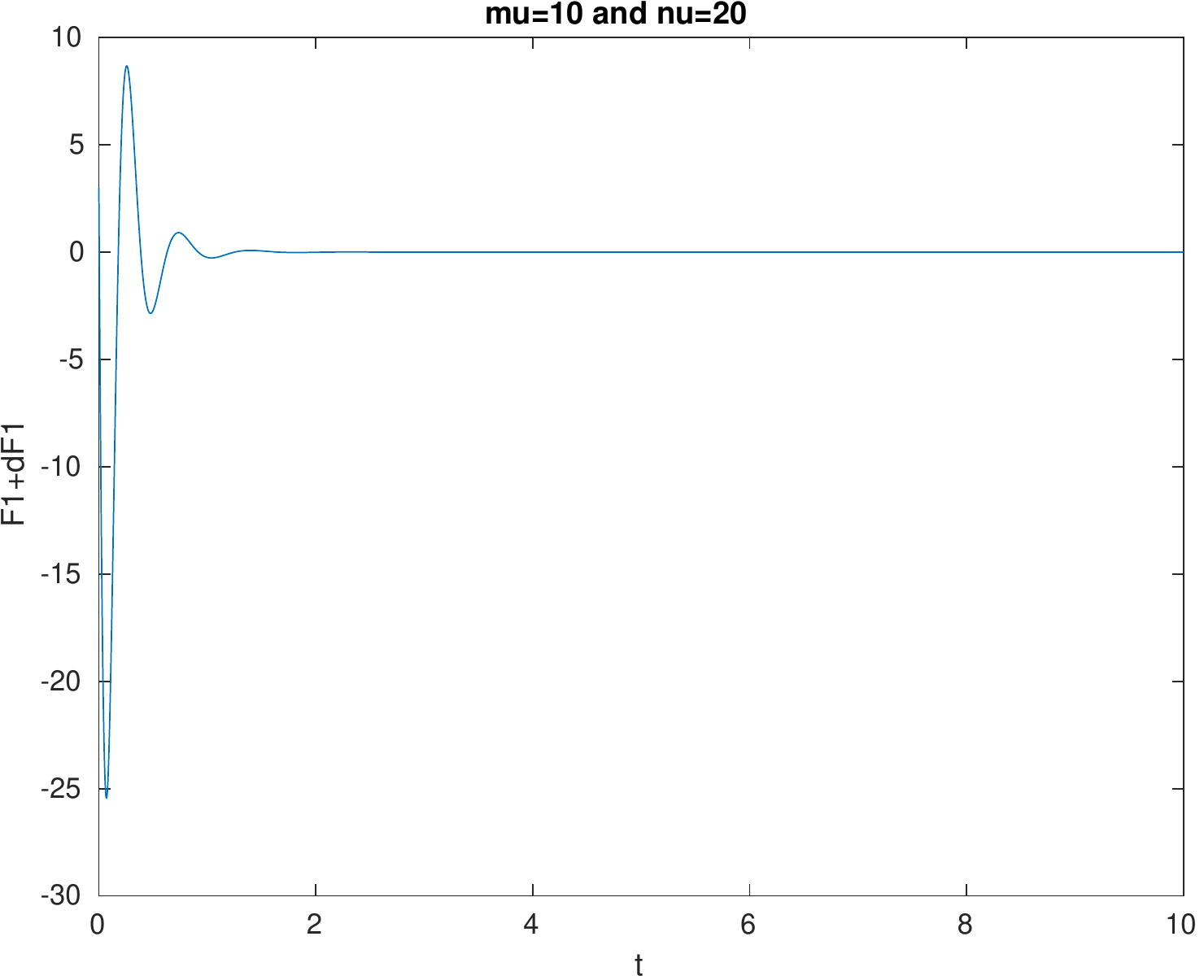} 
    \caption{\tiny{The case $\mu=10, \nu=20$ which corresponds to $\de <0$.}} 
    \vspace{4ex}
  \end{minipage} 
\end{figure}

We end this appendix by stating some observations on the above figures which we believe have the merit to be mentioned:
\begin{itemize}
\item We note that the free-mass case ($\nu=0$) exhibits the positivity of $F_2(t)$, and hence that of $G_2(t)$, for all time starting from  the initial time $t=0$. This is in agreement with our results in \cite{Our2} on the positivity of $F_2(t)$ and $G_2(t)$.
\item From Figures 2, 3 and 4, which correspond to the cases $\de >0, \de =0$ and $\de <0$, respectively, we notice a negative lower bound of $F_2(t)$, but, for large time the functional  $F_2(t)$ is positive. However, more oscillations near $t=0$ are observed when $\de <0$. Of course the case $\de <0$ is not studied in this work but will be the subject of a future investigation.
\end{itemize}


\bibliographystyle{plain}

\end{document}